\documentclass[12pt]{amsart}
\usepackage[margin=1.64in]{geometry}

\usepackage{amssymb, amsmath}
\usepackage{amsthm}
\usepackage{cases}
\usepackage{latexsym}
\usepackage{mathrsfs}
\usepackage{amsfonts}
\usepackage{cases}
\usepackage{enumerate}
\usepackage{caption}
\usepackage{tikz}
\usetikzlibrary{decorations.markings, decorations.pathreplacing, positioning,arrows, matrix, decorations.pathmorphing, shapes.geometric, calc}
\usetikzlibrary{backgrounds, decorations}

\usepackage{cite}
\usepackage[font={it}, margin=1cm]{caption}
\reversemarginpar
 
\newtheorem{theorem}{Theorem}[section]
\newtheorem*{folkconjecture}{Folk Conjecture}
\newtheorem{lemma}[theorem]{Lemma} 
\newtheorem{proposition}[theorem]{Proposition} 

\newtheorem{example}[theorem]{Example}

\newtheorem{question}{Question}

\newtheorem{thmletter}{Theorem}

\newtheorem*{thmIntroduction}{Theorem \ref{thm:TheoremA}}

\newcommand{\out}{\operatorname{Out}}
\newcommand{\aut}{\operatorname{Aut}}

\newcommand{\bfGamma}{\mathbf{\Gamma}}
\newcommand{\bGamma}{\mathbf{\Gamma}}
\newcommand{\<}{\langle}
\renewcommand{\>}{\rangle}
\newcommand{\p}[1]{\noindent {\newline\bf #1.}}

\usepackage{verbatim}

\begin{document}

\title[JSJ-decompositions of one-relator groups with torsion]{The JSJ-decompositions of one-relator groups with torsion}\title{The JSJ-decompositions of one-relator groups with torsion}
\author[A.D.Logan]{Alan D. Logan}

\address{University of Glasgow, School of Mathematics and Statistics, University Gardens, G12 8QW, Scotland}
\email{Alan.Logan@glasgow.ac.uk}
\date{\today}


\maketitle

\begin{abstract}
In this paper we use JSJ-decompositions to formalise a folk conjecture recorded by Pride on the structure of one-relator groups with torsion. We prove a slightly weaker version of the conjecture, which implies that the structure of one-relator groups with torsion closely resemble the structure of torsion-free hyperbolic groups.
\end{abstract}

\section{Introduction}

A classical result of Karrass, Magnus and Solitar states that if a word $R$ over an alphabet $X^{\pm1}$ is not a proper power of any other word over $X^{\pm1}$ then a group $G$ with a presentation of the form $G=\langle X; R^n\rangle$ has torsion if and only if $n>1$; moreover, the word $R$ has order precisely $n$~\cite{karrass1960elements} (see also their book \cite[Corollary~4.11~\&~Theorem~4.12]{mks}). Such a group $G$ is called a \emph{one-relator group with torsion} if we further stipulate that $|X|>1$.

In this paper we study the structure of one-relator groups with torsion which do not admit a free group as a free factor; this is precisely the class of finitely-generated, one-ended one-relator groups with torsion \cite{fischer1972one}.

\p{One-relator groups} One-relator groups are classically studied (see, for example, the classic texts in combinatorial group theory \cite{mks} \cite{L-S} or the early work of Magnus \cite{magnus1930freiheitssatz} \cite{magnus1932wordproblem}), and these groups continue to have applications to this day, for example in $3$-dimensional topology and knot theory~\cite{2014Ichihara}~\cite{Friedl2015Two}. One-relator groups with torsion are hyperbolic and as such they serve as important test cases for this larger class of groups. For example, the isomorphism problem for two-generator, one-relator groups with torsion was shown to be soluble \cite{Pride1977} long before Dahmani--Guirardel's recent resolution of the isomorphism problem for all hyperbolic groups \cite{dahmani2011isomorphism}. As another example, Wise recently resolved the classical conjecture of G. Baumslag that all one-relator groups with torsion are residually finite \cite{wise2012riches}, while it is still an open question as to whether all hyperbolic groups are residually finite.

\p{Main results}
Pride has recorded the following folk conjecture \cite{pride1977twogensubgps}.
\begin{folkconjecture}
Apart from the complications introduced by having elements of finite order, one-relator groups with torsion behave much like free groups.
\end{folkconjecture}
Using JSJ-decompositions, we can make this conjecture concrete as follows. One-relator groups with torsion are hyperbolic, and JSJ-decompositions are important structural invariants of hyperbolic groups. The JSJ-decomposition of a hyperbolic group is a canonical splitting of the group as a graph of groups where every edge group is virtually-$\mathbb{Z}$. For a given hyperbolic group $H$, JSJ-decompositions can be used to determine the model theory~\cite{Sela2009} and the isomorphism class~\cite{dahmani2011isomorphism} of $H$, and to coarsely determine the outer automorphism group~\cite{levitt2005automorphisms} of $H$.
\begin{folkconjecture}[JSJ-version]
Let $G$ be a one-relator group with torsion which does not admit a free group as a free factor. Then precisely one vertex group of the JSJ-decomposition of $G$ contains torsion; all other vertex groups are free groups.
\end{folkconjecture}
The main result of this paper, Theorem~\ref{thm:TheoremA}, proves a slightly weaker version of the above conjecture. We state an abbreviated version of Theorem~\ref{thm:TheoremA} below. The complete statement, along with the proof, can be found in Section~\ref{sec:JSJ}.
\begin{thmIntroduction}
Assume that $G=\<X; R^n\>$ is a one-relator group with torsion which does not admit a free group as a free factor. Then precisely one non-elementary vertex group of the JSJ-decomposition of $G$ contains torsion; all other non-elementary vertex groups are torsion-free hyperbolic groups. In general, an edge group is a subgroup of malnormal infinite cyclic subgroup of $G$.
\end{thmIntroduction}
Theorem~\ref{thm:TheoremA} implies that, in the sense of Dahmani--Guirardel~\cite{dahmani2011isomorphism}, the $\mathcal{Z}$-max JSJ-decomposition of a one-ended one-relator group with torsion $G=\< X;R^n\>$ is precisely the JSJ-decomposition of the group $G$. Dahmani--Guirardel's notion of a $\mathcal{Z}$-max JSJ-decomposition is fundamental to their resolution to the isomorphism problem for all hyperbolic groups.

In this paper we give examples of one-relator groups with torsion which have non-trivial JSJ-decomposition; in each case the torsion-free vertex groups are free, and the vertex group containing torsion is a one-relator group with torsion. These examples imply that the statement of Theorem~\ref{thm:TheoremA} is not vacuous.

\p{Theorem~\ref{thm:TheoremA} and the Folk Conjecture}
Theorem~\ref{thm:TheoremA} implies that the JSJ-decompositions of one-ended one-relator groups with torsion have a structure similar to that of torsion-free hyperbolic groups. This is a weak version of the folk conjecture stated above.

The statement in Theorem~\ref{thm:TheoremA} relating to edge groups is not part of the JSJ-version of the folk conjecture. In Section~\ref{sec:noEvenTor} we state and prove Theorem~\ref{thm:TheoremB}, which is Theorem~\ref{thm:TheoremA} but without the statement about edge groups. Theorem~\ref{thm:TheoremB} is of interest because it has a short proof. Note, however, that the statement about edge groups again mimics the torsion-free hyperbolic groups case, in the sense that in torsion-free hyperbolic groups all edge groups are subgroups of malnormal, infinite cyclic subgroups.

In order to prove the analogous result to Theorem~\ref{thm:TheoremA} for free groups rather than torsion-free hyperbolic groups it would be sufficient to provide a positive answer to the following question.
\begin{question}
\label{Q:freevertices}
Is every torsion-free vertex group in the JSJ-decomposition of a one-ended one-relator group with torsion a free group?
\end{question}
A positive answer to Question~\ref{Q:freevertices} does not completely answer the JSJ-version of the folk conjecture stated above. This is because Theorem~\ref{thm:TheoremA} leaves the possibility that a one-relator group with torsion could split as $G\cong A\ast_CB$ where $B\neq C$ and both $B$ and $C$ are infinite dihedral. A positive answer to Question~\ref{Q:freevertices} and a negative answer to the following question, Question~\ref{Q:infDihedral}, would prove the JSJ-version of the folk conjecture.
\begin{question}
\label{Q:infDihedral}
Does there exist a one-ended one-relator group with torsion which splits non-trivially as $A\ast_CB$ with $B$ and $C$ both infinite dihedral?
\end{question}
Question~\ref{Q:infDihedral} is the least important of the questions stated here. This is because its relevance to the conjecture depends on the precise definition of ``JSJ-decomposition'' we are working with. We discuss this relevance at the end of Section~\ref{sec:JSJ}.

A positive answer to the following question would make the folk conjecture much stronger.
\begin{question}
Is the unique non-elementary vertex group containing torsion in the JSJ-decomposition of a one-ended one-relator group with torsion itself a one-relator group with torsion?
\end{question}
We do not know if these questions admit positive answers or not. Note that B.B.Newman provided a possible counter-example to the folk conjecture which motivates this paper \cite{pride1977twogensubgps}, and one can interpret Newman's construction as a possible counter-example to Question~\ref{Q:freevertices}. However, in Example~\ref{ex:2-genLoop} from Section~\ref{sec:examples} we prove that Newman's group is not a counter-example to Question~\ref{Q:freevertices}, and so not a counter-example to the JSJ-version of the folk conjecture.

\p{Overview of the paper}
In Section~\ref{sec:JSJDefn} we motivate and define JSJ-decompositions of hyperbolic groups. In Section~\ref{sec:onerelbground} we prove a key preliminary result, Lemma~\ref{lem:ends}, that relates the number of ends of a one-relator group with torsion $\<X; R^n\>$ to the number of ends of the associated torsion-free one-relator group $\<X; R\>$. In Section~\ref{sec:onerelsubgroups} we recall certain fundamental results on the subgroups of one-relator groups with torsion. In Section~\ref{sec:JSJ} we use the results of Sections~\ref{sec:onerelbground}~and~\ref{sec:onerelsubgroups}, as well as standard results on amalgams and HNN-extensions, to classify how a one-ended one-relator group can split as an HNN-extension (Lemma~\ref{lem:HNN}) or amalgam (Lemma~\ref{lem:amalg}) over a virtually-$\mathbb{Z}$ subgroup. These two lemmas combine to prove the Principal Lemma, Lemma~\ref{thm:generalJSJstructural}, which yields Theorem~\ref{thm:TheoremA}. In Section~\ref{sec:2Gen1Rel} we outline two results on the structure of two-generator, one-relator groups with torsion. These are Theorem~\ref{thm:JSJgraph}, which describes explicitly the possible JSJ-decompositions of such a group, and Theorem~\ref{thm:JSJdecomp}, which proves that such a group is either of the form $\< a, b; [a, b]^n\>$ or has virtually-cyclic outer automorphism group.
In Section~\ref{sec:examples} we give examples of one-relator groups with torsion which have non-trivial JSJ-decomposition. These examples imply that the results of this paper are not vacuous.


\section{JSJ-decompositions: motivation and definition}
\label{sec:JSJDefn}

JSJ-decompositions are important structural invariants of hyperbolic groups. One-relator groups with torsion are hyperbolic~\cite[Theorem IV.5.5]{L-S}, and this paper studies the JSJ-decompositions of one-relator groups with torsion. Therefore, in this section we motivate and give the definition of the JSJ-decomposition of a one-ended hyperbolic group. JSJ-decompositions encode the virtually-$\mathbb{Z}$ splittings of hyperbolic groups, so in Section~\ref{sec:JSJ} we investigate the virtually-$\mathbb{Z}$ splittings of one-ended one-relator groups with torsion, and this analysis yields the main result of this paper, Theorem~\ref{thm:TheoremA}.

Note that the definition of a ``one-ended hyperbolic group'' is not necessary to understand this paper, because a one-relator group with torsion is one-ended if and only if it does not admit a free group as a free factor~\cite{fischer1972one}. This observation also means that our restriction to the one-ended case does not weaken our solution to the folk-conjecture which motives this paper.

JSJ-decompositions are a certain kind of graph of groups splitting. We use $\bGamma$ to denote a graph of groups, with a connected underlying graph $\Gamma$, associated \emph{vertex groups} (or \emph{vertex stabilisers}) $\{G_v; v\in V_{\Gamma}\}$ and \emph{edge groups} (or \emph{edge stabilisers}) $\{G_e; e\in E_{\Gamma}, G_e=G_{\overline{e}}\}$, and set of monomorphisms $\theta_e: G_e\rightarrow G_{\iota(e)}$ \cite{trees}.

\p{Background on JSJ-decompositions} We now motivate the study of, and loosely define, JSJ-decompositions of hyperbolic groups. A more careful definition, due to Bowditch \cite{bowditch1998cut}, then follows.

Mirroring the JSJ-decomposition of $3$-manifolds \cite{hatcher2000notes}, the JSJ-decomposition of a one-ended hyperbolic group is a graph of groups decomposition where all edge groups are virtually-$\mathbb{Z}$, and is (in an appropriate sense) unique if certain vertices, called ``orbifold vertices'', are treated as pieces to be left intact and not decomposed \cite{bowditch1998cut}. Sela used these JSJ-decompositions to obtain results regarding the outer automorphism groups of torsion-free hyperbolic groups \cite{rips1997cyclic}, regarding the isomorphism problem for torsion-free hyperbolic groups \cite{sela1995isomorphism}, and regarding the model theory of all hyperbolic groups\cite{Sela2009}. Levitt used JSJ-decompositions to coarsely classify the outer automorphism groups of hyperbolic groups \cite{levitt2005automorphisms}.

Dahmani--Guirardel defined and used \emph{$\mathcal{Z}$-max JSJ-decompositions}, a modification of JSJ-decompositions where edge groups are virtually-$\mathbb{Z}$ with infinite center, to resolve the isomorphism problem for all hyperbolic groups \cite{dahmani2011isomorphism}. The main result of our paper, Theorem~\ref{thm:TheoremA}, implies that the $\mathcal{Z}$-max JSJ-decomposition of a one-ended one-relator group with torsion $G$ is simply the JSJ-decomposition of $G$.

\p{The definition of JSJ-decompositions} We shall now define the canonical JSJ-decomposition of a hyperbolic group. First, however, we define certain terms used in the definition. An \emph{elementary subgroup} of a hyperbolic group is a virtually-$\mathbb{Z}$ subgroup. A \emph{refinement} of a graph of groups is a splitting $\bfGamma_v$ of a vertex group $G_v$ over a subgroup $C\leq G_v$ which respects the graph of groups structure \cite{rips1997cyclic}, in the sense that images of adjacent edge groups in $G_v$ are
vertex groups in the splitting $\bfGamma_v$; if such a splitting exists then we say that $G_v$ can be \emph{refined over $C$}.

Let $G$ be a one-ended hyperbolic group which is not a Fuchsian group. Then a \emph{JSJ-decomposition} of $G$ is a splitting of $G$ as a graph of groups $\bGamma$ with three types of vertices.
\begin{enumerate}
\item\label{enum:elem} \emph{Elementary vertices}, whose group is a maximal elementary subgroup.
\item\label{enum:orb} \emph{Orbifold vertices}, whose group is a ``maximal hanging Fuchsian subgroup''.\footnote{We only need orbifold vertices to define JSJ-decompositions; they are not mentioned anywhere else in this paper. Therefore, the interested reader is referred to Bowditch \cite{bowditch1998cut} for a formal definition.}
\item\label{enum:rig} \emph{Rigid vertices}, whose group cannot be refined over an elementary subgroup.
\end{enumerate}
Every edge connects an elementary vertex to either a rigid or an orbifold vertex. Edge groups of orbifold vertices correspond to the peripheral subgroups of the hanging Fuchsian group. Finally, edge groups are maximal elementary in the corresponding rigid or orbifold vertex group.

A JSJ-decomposition of a hyperbolic group is unique up to conjugacy\footnote{The precise equivalence relation can be found in Bowditch's paper \cite[Section~6]{bowditch1998cut}.}\cite{bowditch1998cut} and so we can talk about \emph{the} JSJ-decomposition of a hyperbolic group.

\section{The ends of one relator groups}
\label{sec:onerelbground}

In this section we prove Lemma~\ref{lem:ends}, which relates the number of ends of a one-relator group with torsion $G=\<X; R^n\>$ to the structure of the related one-relator group \emph{without} torsion $\widehat{G}=\<X; R\>$. Lemma~\ref{lem:ends} is a key technical result in this paper and is applied in Lemma~\ref{lem:HNN} and Lemma~\ref{lem:amalg} from Section~\ref{sec:JSJ}, which are the two cases of the Principal Lemma, Lemma~\ref{thm:generalJSJstructural}.

\p{The group $\mathbf{\widehat{G}}$} Suppose $G$ is an arbitrary group. We define the group $\widehat{G}$ as the quotient of $G$ by the normal closure $T$ of the torsion elements, that is, $\widehat{G}:=G/T$. The following proposition is due to Karrass, Magnus and Solitar~\cite{karrass1960elements}. Recall that $R$ denotes a word which is not a proper power of any other word over $X^{\pm1}$.
\begin{proposition}
\label{prop:Ghat}
In a one relator group with torsion $G=\<X; R^n\>$ the group $\widehat{G}$ has presentation $\<X; R\>$ and is torsion free.
\end{proposition}
Lemma~\ref{lem:ends} relates the ends of $G=\<X; R^n\>$, $n>1$, to the structure of the group $\widehat{G}=\<X; R\>$.

\p{The ends of one-relator groups with torsion} A finitely-generated infinite group can have one, two or infinitely many ends, and the class of two-ended groups is precisely the class of virtually-$\mathbb{Z}$ groups. In this section we use the fact that a finitely-generated one-relator group with torsion has one end if and only if it does not admit a free group as a free factor, and has infinitely many ends otherwise~\cite{fischer1972one}. We also use the fact that a finitely-generated torsion-free group $G$ is infinitely ended if and only if it decomposes non-trivially as a free product, so $G=H\ast K$ where $H$ and $K$ are both non-trivial~\cite{stallings1968torsion}.

Lemma~\ref{lem:ends} generalises the following observation: a two-generator, one-relator group with torsion $G=\< a, b; R^n\>$ is infinitely ended if and only if the relator $R$ is contained in some free basis of $F(a, b)$, and $G$ is one-ended otherwise \cite{Pride1977}. Our generalisation uses the natural surjection from $G=\< X; R^n\>$, $n>1$, to $\widehat{G}=\< X; R\>$. The kernel of this map is the normal closure of the element $R$, $T=\<\< R\>\>$, by Proposition~\ref{prop:Ghat}. This subgroup $T=\<\<R\>\>$ plays an important role in Section~\ref{sec:JSJ}, where we prove the main results of this paper.

In the proof of Lemma~\ref{lem:ends} we use the term ``group with a single defining relation'' to mean a group with a presentation of the form $\<X; S\>$ where $|X|\geq1$, as opposed to $|X|>1$.

\begin{lemma}
\label{lem:ends}
Assume that $G=\<X; R^n\>$ is a finitely-generated one-relator group with torsion. Then $G$ is infinitely ended if and only if $\widehat{G}$ is either infinitely ended or infinite cyclic. Otherwise both $G$ and $\widehat{G}$ are one-ended.
\end{lemma}

\begin{proof}
Let $G=\< x_1, x_2, \ldots, x_m; R^n\>$ be a one-relator group with torsion and without loss of generality we can assume that $R^n$ has minimal length in its $\aut(F_m)$ orbit.

Now, $G$ and $\widehat{G}$ are infinite (as we assume $m> 1$), so $G$ has either one-end or is a free product of a non-trivial free group and an indecomposable group with a single defining relation \cite{fischer1972one}.

Suppose $G$ is a free product of a nontrivial free group and an indecomposable group with a single defining relation. Then, relabeling the generators if necessary, there exists some $1\leq k<m$ such that $G\cong G_1\ast G_2$ where $G_1=\<x_1, \ldots, x_k; R^n\>$ is freely indecomposable and $G_2$ is free with basis $x_{k+1}, \ldots, x_m$ \cite[Proposition II.5.13]{L-S}. Thus, $\widehat{G}=\widehat{G}_1\ast{G}_2$. If $\widehat{G}_1$ is non-trivial then $\widehat{G}$ is infinitely ended. If $\widehat{G}_1$ is trivial then $\widehat{G}$ is free, and so is either infinite cyclic or has infinitely many ends, as required.

In order to prove the lemma, it is now sufficient to prove that if $\widehat{G}$ is infinitely ended or infinite cyclic then $G$ is infinitely ended. So, firstly suppose that $G=\< x_1, x_2, \ldots, x_m; R^n\>$, $n>1$, is one-ended but $\widehat{G}$ is infinitely ended. As $\widehat{G}$ is a one-relator group (note that it is possibly isomorphic to a free group) it can be decomposed as $G_1\ast G_2$ where $G_1=\<x_1, \ldots, x_k; R\rangle$ is a freely indecomposable group (possible trivial) with a single defining relation, and where $G_2$ is free. Because $R^n$ has minimal length in its $\aut(F_m)$ orbit, so has $R$. We can then apply the fact that $G$ is one-ended to get that $G_2$ is trivial, and so $\widehat{G}=G_1$ is freely indecomposable.
As $\widehat{G}$ is torsion-free and freely indecomposable it is not infinitely ended, a contradiction.
Secondly, suppose that $G=\<x_1, \ldots, x_m; R^n\>$, $n>1$, is one-ended but $\widehat{G}$ is two-ended, and again $R$ has minimal length in its $\aut(F_m)$ orbit. Then $\widehat{G}$ is free of rank one, because the only two-ended torsion-free group is the infinite cyclic group. Therefore, $R$ is primitive \cite[Proposition II.5.10]{L-S}. As $R$ has minimal length in $\aut(F_m)$, $G$ has presentation $\<x_1, \ldots, x_m; x_1^n\>$ and so cannot be one-ended, a contradiction.
\end{proof}

\section{The subgroups of one relator groups}
\label{sec:onerelsubgroups}

The purpose of this section is to state certain known results on the subgroups of one-relator groups with torsion which we apply in Section~\ref{sec:JSJ}. We also prove Lemma~\ref{lem:vcsubgroups}, which shows that, apart from the complications introduced by having elements of finite order, the virtually-$\mathbb{Z}$ subgroups of one-relator groups with torsion behave much like those of torsion-free hyperbolic groups (equivalently, of free groups). JSJ-decompositions encode the possible virtually-$\mathbb{Z}$ splittings of a hyperbolic group, and thus Lemma~\ref{lem:vcsubgroups} is a key technical result in this paper. We apply Lemma~\ref{lem:vcsubgroups} in Lemmas~\ref{lem:HNN}~and~\ref{lem:amalg} from Section~\ref{sec:JSJ}, which are the two cases of the Principal Lemma, Lemma~\ref{thm:generalJSJstructural}.

We begin by stating three known results, Propositions~\ref{prop:sbgpT},~\ref{prop:conjToR}~and~\ref{prop:vZsubgpsareZorD}, on the subgroups of one-relator groups with torsion. These results are used throughout Section~\ref{sec:JSJ}, while Proposition~\ref{prop:vZsubgpsareZorD} is also used in the proof of Lemma~\ref{lem:vcsubgroups}.

\p{The subgroup \boldmath{$T$}} Let $G=\langle X; R^n\rangle$ be a one-relator group with torsion. The first result we state relates to the structure of the subgroup $T:=\<\<R\>\>$, as introduced in Section~\ref{sec:onerelbground}. This subgroup $T$ plays an important role in Section~\ref{sec:JSJ}. The result is due to Fischer, Karrass and Solitar~\cite{fischer1972one}.
\begin{proposition}[Fischer--Karrass--Solitar, 1972]
\label{prop:sbgpT}
Assume that $G=\langle X; R^n\rangle$ is a one-relator group with torsion. Then the subgroup $T:=\langle\langle R\rangle\rangle$ is isomorphic to the free product of infinitely many cyclic groups of order~$n$.
\end{proposition}

The second result we state implies that the finite cyclic free factors of $T$ are conjugate in $G$, although clearly they are not conjugate in the abstract group $T$. This conjugate/non-conjugate duality is applied in Lemmas~\ref{lem:ZsplitFPA}~and~\ref{lem:DsplitFPA} from Section~\ref{sec:JSJ}. These two lemmas combine to prove Lemma~\ref{lem:amalg}, which is one of the two cases of the Principal Lemma. The result is due to Karrass, Magnus and Solitar~\cite{karrass1960elements}.
\begin{proposition}[Karrass--Magnus--Solitar, 1960]
\label{prop:conjToR}
Assume that $G=\langle X; R^n\rangle$ is a one-relator group with torsion. If $g\in G$ is an element of finite order in $G$ then $g$ is conjugate to some power of the element $R$, so $h^{-1}gh=R^i$.
\end{proposition}

\p{Virtually-$\mathbb{Z}$ subgroups} The third result we state classifies the isomorphism classes of the virtually-$\mathbb{Z}$ subgroups of a one-relator group with torsion. The result is due independently to Karrass and Solitar~\cite{karrass1971subgroups}, and to Cebotar~\cite{cebotar1971subgroups}.

\begin{proposition}[Karrass--Solitar, Cebotar, 1971]
\label{prop:vZsubgpsareZorD}
If $C$ is a virtually-$\mathbb{Z}$ subgroup of a one-relator group with torsion then $C$ is either infinite cyclic or infinite dihedral.
\end{proposition}

We now prove Lemma~\ref{lem:vcsubgroups}, which is a key technical result of this paper. The result classifies, using malnormality, those virtually-$\mathbb{Z}$ subgroups which are not contained in any infinite dihedral subgroup of $G$.
We will take $h^g$ to mean conjugation by $g$, $h^g:=g^{-1}hg$. A subgroup $M$ of a group $G$ is said to be \emph{malnormal} in $G$ if $M^g\cap M\neq 1$ implies that $g\in M$.

\begin{lemma}
\label{lem:vcsubgroups}
Assume that $C$ is a virtually-$\mathbb{Z}$ subgroup of a one-relator group with torsion $G$. Suppose that $C$ is not a subgroup of an infinite dihedral group. Then $C$ is a subgroup of a malnormal, infinite cyclic subgroup of $G$.
\end{lemma}

\begin{proof}
As $G$ is hyperbolic every infinite cyclic subgroup of $G$ is contained in a unique maximal virtually-$\mathbb{Z}$ subgroup \cite{bowditch1998cut}. Moreover, in a hyperbolic group a maximal virtually-$\mathbb{Z}$ subgroup of $G$ which is torsion free is malnormal in $G$.
The result then follows from Proposition~\ref{prop:vZsubgpsareZorD}.
\end{proof}

\section{The JSJ-decompositions of one relator groups}
\label{sec:JSJ}
\label{sec:virtZsplit}

The results of this section show that, apart from the complications introduced by having elements of finite order, the possible JSJ-decompositions of one-ended one-relator groups with torsion are much like those of one-ended torsion-free hyperbolic groups.

More formally, in the Principal Lemma, Lemma~\ref{thm:generalJSJstructural}, we prove that, apart from possibly certain specific cases, every edge group of the JSJ-decomposition of a one-ended one-relator group with torsion $G$ is a subgroup of a malnormal infinite-cyclic group (this is precisely what happens in torsion-free hyperbolic groups). From this, Theorem~\ref{thm:TheoremA} follows quickly.

\begin{thmletter}
\label{thm:TheoremA}
Assume that $G=\<X; R^n\>$ is a one-ended one-relator group with torsion. Then precisely one non-elementary vertex group of the JSJ-decomposition of $G$ contains torsion; all other non-elementary vertex groups are torsion-free hyperbolic groups. If $v$ is an elementary vertex in the JSJ-decomposition of $G$ then one of the following occurs.
\begin{enumerate}
\item\label{item:EdgeGpStatement} The vertex $v$ has arbitrary degree, and the vertex group $G_v$ is a subgroup of a malnormal, infinite cyclic subgroup of $G$.
\item The vertex $v$ has degree one, and both the vertex group $G_v$ and the adjacent edge group $G_e$ are infinite dihedral.
\end{enumerate}
\end{thmletter}

\p{Proving the Principal Lemma} Theorem~\ref{thm:TheoremA} follows quickly from the Principal Lemma, Lemma~\ref{thm:generalJSJstructural}, and the proof of the Principal Lemma takes up the majority of this section, Section~\ref{sec:JSJ}. The issue that the lemma resolves is that, by Proposition~\ref{prop:vZsubgpsareZorD}, the edge groups in the JSJ-decomposition may be infinite dihedral. In order to prove the Principal Lemma we prove that if a one-ended, one-relator group with torsion $G=\<X; R^n\>$ splits as an amalgam or HNN-extension over a virtually-$\mathbb{Z}$ subgroup $C$ with non-virtually-$\mathbb{Z}$ edge group(s) then, in general, $C$ cannot be contained in the normal closure $T:=\<\<R\>\>$ of the element $R$, and so, by Lemma~\ref{lem:vcsubgroups}, $C$ is contained in a malnormal, infinite cyclic subgroup of $G$. Lemma~\ref{lem:HNN} proves the result for HNN-extensions while Lemma~\ref{lem:amalg} proves it for free products with amalgamation.

\subsection{No even torsion}
\label{sec:noEvenTor}

A weakened form of Theorem~\ref{thm:TheoremA} admits a quick proof for one-ended one-relator groups with torsion which contain no even torsion. We state and prove this result, Theorem~\ref{thm:TheoremB}, here. Note that Theorem~\ref{thm:TheoremB} is only of interest because it admits a short proof: Theorem~\ref{thm:TheoremA} is a stronger result which holds for these groups.

\begin{thmletter}
\label{thm:TheoremB}
Assume that $G=\<X; R^n\>$ is a one-ended one-relator group with torsion, and suppose $n$ is odd. Then precisely one non-elementary vertex group of the JSJ-decomposition of $G$ contains torsion; all other non-elementary vertex groups are torsion-free hyperbolic groups.
\end{thmletter}

\begin{proof}
Every edge group of the JSJ-decomposition of such a group $G$ is infinite-cyclic, by Proposition~\ref{prop:vZsubgpsareZorD}. Note that $R$ is contained in a conjugate of a vertex group $G_v$ in the JSJ-decomposition of $G$, and suppose that $g$ has finite order in $G$ but is contained in a different vertex group of the JSJ-decomposition, so $g\in G_w$ with $v\neq w$. Then, by Proposition~\ref{prop:conjToR}, $g$ is conjugate to an element of $G_v$, so $h^{-1}gh\in G_v$. Hence, $g\in G_{hv}$ with $hv\neq w$. However, then $g$ fixes the geodesic between $w$ and $hv$, and hence some edge stabiliser contains torsion, a contradiction.
\end{proof}

\subsection{The proof of the Principal Lemma}

Theorem~\ref{thm:TheoremA} follows quickly from the Principal Lemma, Lemma~\ref{thm:generalJSJstructural}. We begin this section with Lemma~\ref{lem:HNN}, which proves the Principal Lemma for HNN-extensions. The proof of this case is short. The case of free products with amalgamation is more involved, and the subsequent four lemmas,
Lemmas~\ref{lem:AsbgpT}--\ref{lem:amalg}, prove this case of the Principal Lemma. We then state and prove the Principal Lemma.

Recall that if $G$ is a group and $T$ is the normal closure of its torsion elements then we write $\widehat{G}:=G/T$. If $G=\<X; R^n\>$ is a one-relator group with torsion then $T:=\<\<R\>\>$ and $\widehat{G}\cong\<X; R\>$ is a torsion-free one-relator group, by Proposition~\ref{prop:Ghat}.

\p{HNN-extensions} We shall now prove the Principal Lemma for HNN-extensions.

\begin{lemma}
\label{lem:HNN}
Assume that $G=\<X; R^n\>$ is a one-ended one-relator group with torsion which splits as an HNN-extension $G\cong H\ast_{A^t=B}$ where $A$ and $B$ are virtually-$\mathbb{Z}$ groups. Then $A, B\not\leq T=\<\<R\>\>$ and so $A$ and $B$ are subgroups of malnormal, infinite cyclic subgroups of $G$.
\end{lemma}

\begin{proof}
Suppose, without loss of generality, that $A\leq T$. We further have that $B\leq T$, as $G=\langle H, t; A^t=B\rangle$ so $B$ is contained in the normal closure of $A$. On the other hand, $t\not\in T$ as $T$ is the normal closure of a torsion element, and so of an element contained in (a conjugate of) $H$. Therefore, $\widehat{G}$ is the free product $\widehat{H}\ast\langle t\rangle$ where $\widehat{H}$ is obtained from $H$ by quotienting out the normal closure of the torsion elements of $H$. Thus, by Lemma~\ref{lem:ends}, $G$ is infinitely ended, a contradiction. So $A, B\not\leq T$ and the result follows from Lemma~\ref{lem:vcsubgroups}.
\end{proof}

\p{Amalgams} We prove, in Lemma~\ref{lem:amalg}, the analogous result to Lemma~\ref{lem:HNN} for free products with amalgamation. Lemma~\ref{lem:amalg} states that if $G$ is a one-ended one-relator group with torsion and if $G=A\ast_CB$ where $C$ is virtually-$\mathbb{Z}$ then either $C\not\leq T$, as in Lemma~\ref{lem:HNN}, or $C$ and one of $A$ or $B$ is infinite dihedral.
The proof of Lemma~\ref{lem:amalg} comprises Lemma~\ref{lem:AsbgpT}, which, supposing $C\leq T$, gives a form for one of the factor groups $A$ or $B$, Lemma~\ref{lem:ZsplitFPA}, which investigates when $C$ is infinite cyclic, and Lemma~\ref{lem:DsplitFPA}, which investigates when $C$ is infinite dihedral.

Note that Lemmas~\ref{lem:AsbgpT}~and~\ref{lem:ZsplitFPA} include the case of $\mathbb{Z}\ast C_n$. This inclusion allows methods of Kapovich--Weidmann to be applied to obtain a coarse description of the outer automorphism groups of two-generator, one-relator groups with torsion. This application is outlined in Section~\ref{sec:2Gen1Rel}.

We begin our proof of Lemma~\ref{lem:amalg} by giving a form for one of the factor groups $A$ or $B$ of $G=A\ast_CB$ when the the amalgamating subgroup $C$ is subgroup of $T:=\<\< R\>\>$. A splitting $A\ast_CB$ is called \emph{non-trivial} if $C\lneq A, B$.

\begin{lemma}
\label{lem:AsbgpT}
Assume that $G=\<X; R^n\>$ is a one-relator group with torsion which is either one-ended or isomorphic to $\mathbb{Z}\ast C_n$. Suppose that $G$ splits non-trivially as a free product with amalgamation $G\cong A\ast_{C}B$ where $C$ is a subgroup of $T=\<\<R\>\>$. Then either $A\leq T$ or $B\leq T$. If $A\leq T$ (the case of $B\leq T$ is analogous) then $A\cong F_m\ast A_{1}\ast A_{2}\ast\ldots$ with each $A_i=\<a_i\>$ non-trivial cyclic of order $n_i$ dividing $n$ and $F_m$ is free of rank $m\geq 0$. There may be only finitely many $A_i$, and indeed there may be none.
\end{lemma}

\begin{proof}
Let $\widehat{A}$ and $\widehat{B}$ be the images of $A$ and $B$ in $\widehat{G}=G/T$. Note that the image $\widehat{C}$ of $C$ in $\widehat{G}$ is trivial, $\widehat{C}=1$, and so $\widehat{G}=\widehat{A}\ast \widehat{B}$. If $G$ is one-ended then one of the factors $\widehat{A}$ or $\widehat{B}$ must be trivial, by Lemma~\ref{lem:ends}. If $G\cong\mathbb{Z}\ast C_n$ then $\widehat{G}$ is infinite cyclic so again one of the factors $\widehat{A}$ or $\widehat{B}$ must be trivial. Thus, either $A\leq T$ or $B\leq T$.

As $T$ is the free product of cyclic subgroups of order $n$, by Proposition~\ref{prop:sbgpT}, the result follows from the Kurosh Subgroup Theorem.
\end{proof}

Now, if $G=A\ast_CB$ is a one-relator group with torsion and the amalgamating subgroup $C$ is virtually-$\mathbb{Z}$ then either $C$ is infinite cyclic or infinite dihedral, by Proposition~\ref{prop:vZsubgpsareZorD}. We shall now, in Lemma~\ref{lem:ZsplitFPA}, investigate the case when $C$ is infinite cyclic, while in Lemma~\ref{lem:DsplitFPA}, below, we investigate the case when $C$ is infinite dihedral.

We shall write $C_A$ (resp. $C_B$) for the copy of $C$ in $A$ (resp. $B$), and so $G=A\ast_{C_A=C_B}B$.
If $A$ is a (not necessarily proper) subgroup of a group $G$ we define the \emph{$A$-normal closure} of a set $S\subset A$, denoted $\<\<S\>\>_A$, to be the normal closure of the set $S$ in the abstract group $A$, as opposed to in $G$. Recall that a splitting $A\ast_CB$ is called {non-trivial} if $C\lneq A, B$.

\begin{lemma}
\label{lem:ZsplitFPA}
Assume that $G=\<X; R^n\>$ is a one-relator group with torsion which is either one-ended or isomorphic to $\mathbb{Z}\ast C_n$. Suppose that $G$ splits non-trivially as a free product with amalgamation $G\cong A\ast_{C}B$ where $C$ is infinite cyclic. Then $C$ is not a subgroup of $T=\<\<R\>\>$.
\end{lemma}

\begin{proof}
Suppose that $C$ is infinite cyclic and contained in $T$, and we shall find a contradiction. As $C\leq T$, we can apply Lemma~\ref{lem:AsbgpT} to get that, without loss of generality, $A\leq T$ and $A\cong F_m\ast A_{1}\ast A_{2}\ast\ldots$ with each $A_i=\<a_i\>$ non-trivial cyclic of order $n_i$ dividing $n$ and $F_m$ is free of rank $m\geq0$. We have two cases: either the root $R$ of the relator $R^n$ is contained in a conjugate of $A$ or is contained in a conjugate of $B$.

Let us consider the second case: Suppose $R\in g^{-1}Bg$. If $A$ contains torsion, so $A_1=\<a_1\>$ is non-trivial, then $a_1$ is conjugate in $G$ to a power of $gRg^{-1}\in B$, by Proposition~\ref{prop:conjToR}, and by the conjugacy theorem for free products with amalgamation \cite[Theorem 4.6]{mks} we have that $a_1$ is conjugate to an element of $C_A$, a contradiction as $C_A$ is torsion-free.
If $A$ is torsion-free then $A\cong F_m$ is the $A$-normal closure of $C_A\cong\mathbb{Z}$, as $A\leq T$, and so the abstract group $A$ is the normal closure of a single element. As a group with more generators than relators is infinite, we have that $m=1$, and indeed $C_A=A$, a contradiction. Hence, if $C_A\cong C$ is infinite cyclic then $R$ cannot be contained in a conjugate of $B$.

Let us consider the first case: Suppose that $R\in g^{-1}Ag$, and by re-writing $R$ we can assume that $R\in A$ and indeed that $R=a_1$, where $A_1=\<a_1\>$. If $A_i$, $i>1$, is non-trivial then the generator $a_i$ of $A_i$ is conjugate to a power of $a_1$, by Proposition~\ref{prop:conjToR}, and so there exists some $k\in\mathbb{Z}$ such that $a_1^k$ and $a_i$ are conjugate in $G$ but not in $A$. Therefore, by the conjugacy theorem for free products with amalgamation, $a_1^k$ and $a_i$ are both contained in conjugates of the amalgamating subgroup $C_A$, a contradiction as $C_A$ is torsion-free. Thus, $A\cong F_m\ast A_1$ where $F_m$ is free of rank $m\geq 1$ and $A_1$ is finite cyclic ($F_m$ cannot be trivial as $G$ is one-ended).

Now, as the subgroup $T$ is the $G$-normal closure of $R=a_1$, and because $A$ is a subgroup of $T$, we have that $C_A$ intersects the $A$-normal closure of $a_1$ non-trivially, $\<\<a_1\>\>_A\cap C_A\neq 1$. Then, as $A/\<\<a_1\>\>_A\cong F_m$ is torsion-free and because $C_A$ is infinite cyclic, we have that $C_A\leq\<\<a_1\>\>_A$. On the other hand, as $A$ is contained in the $G$-normal closure of $a_1$ we have that the $A$-normal closure of $a_1$ and $C_A$ must be the whole of $A$, $\<\<a_1, C_A\>\>_A=A$. Then, as $\<\< a_1, C_a\>\>_A=\<\<a_1\>\>_A$, the group $A\cong F_m\ast A_1$ is the normal closure of a single element, and so $m=0$, a contradiction. Hence, if $C_A\cong C$ is infinite cyclic then $R$ cannot be contained in a conjugate of $A$.

Combining the cases, we conclude that if $C$ is infinite cyclic then $R$ cannot be contained in a conjugate of $A$ or of $B$, a contradiction.
\end{proof}

We shall now analyse how a one-ended one-relator group with torsion can split as a free-product with amalgamation over an infinite dihedral group. This, combined with Lemma~\ref{lem:ZsplitFPA}, shall prove Lemma~\ref{lem:amalg}, which proves the case of free products with amalgamation in the Principal Lemma.

If $A$ is a (not necessarily proper) subgroup of a group $G$ we say two elements $g, h\in A\leq G$ are \emph{$A$-conjugate} if there exists an element $k\in A$ such that $k^{-1}gk=h$.
Recall that a splitting $A\ast_CB$ is called {non-trivial} if $C\lneq A, B$.
We shall again use $C_A$ to denote the copy of the amalgamating subgroup $C$ contained in the factor group $A$.

\begin{lemma}
\label{lem:DsplitFPA}
Assume that $G=\<X; R^n\>$ is a one-ended one-relator group with torsion. Suppose that $G$ splits non-trivially as a free product with amalgamation $G\cong A\ast_{C}B$ where $C$ is infinite dihedral. Then either $A$ or $B$ is infinite dihedral.
\end{lemma}

\begin{proof}
As $C$ is infinite dihedral it is generated by two elements of order two, $C\cong C_2\ast C_2$, and so $C\leq T$. We can then apply Lemma~\ref{lem:AsbgpT} to get that, without loss of generality, $A\leq T$ and $A\cong F_m\ast A_{1}\ast A_{2}\ast\ldots$ with each $A_i=\<a_i\>$ non-trivial cyclic of order $n_i$ dividing $n$ and $F_m$ is free of rank $m$. Now, because $A\leq T$ and $T$ is the normal closure of the torsion elements, the subgroup $A$ is equal to the $A$-normal closure of $C_A$ along with the finite cyclic factors, $A=\<\< C_A, A_1, A_2, \ldots \>\>_A$, and so is the $A$-normal closure of torsion elements. Hence, $m=0$. Note that both $A_1=\<a_1\>$ and $A_2=\<a_2\>$ must be non-trivial as otherwise $A$ is finite cyclic. Suppose, without loss of generality, that $A_1$ is of maximal order in the free-factor groups $A_i$.

In order to prove the lemma it is sufficient to prove that $A$ is infinite dihedral. To prove this we shall use the fact that, for all $i>1$, $a_1$ and $a_i$ are both conjugates of powers of the element $R$, by Proposition~\ref{prop:conjToR}. There are two cases: either the root $R$ of the relator $R^n$ is contained in a conjugate of $A$ or is contained in a conjugate of $B$.

Let us consider the second case: Suppose $R\in g^{-1}Bg$. Then, by the conjugacy theorem for free products with amalgamation, we have that $A$-conjugates of each $a_i$ are contained in the free-factor group $C_A$, and so $a_1$ and each $a_i$, $i>1$, have order two. This observation also implies that $A$ has no more than two $A$-conjugacy classes of elements of finite order, because $C_A\cong D_{\infty}$ has precisely two conjugacy classes of elements of finite order. Now, suppose $A$ is not infinite dihedral. Then $A_3=\<a_3\>$ is non-trivial, and as $A$ is a free product with $a_1$, $a_2$ and $a_3$ in different free factors we have that these three elements are pairwise non-conjugate in $A$. Hence, $A$ has at least three $A$-conjugacy classes of elements of finite order, which is a contradiction. Thus, $A$ is infinite dihedral.

Let us consider the first case: Suppose $R\in g^{-1}Ag$. Then we can re-write $R$ to get that $R\in A$, and indeed that $R=a_1$. Therefore, each $a_i$, $i>1$, is a $G$-conjugate of a power of $a_1$, $a_1^{k_i}$ say, but not an $A$-conjugate of $a_1^{k_i}$. Thus, applying the conjugacy theorem for free products with amalgamation, we have that an $A$-conjugate of $a_1^{n/2}$ is contained in $C_A$ and an $A$-conjugate of each $a_i$ is contained in $C_A$. As with the previous case, applying the fact that the infinite dihedral group has two conjugacy classes of elements of order two yields that  $A=A_1\ast A_2$ with $A_2\cong C_2$, and here $A_1\cong C_n$.

To complete the proof of this case it is sufficient to prove that $n=2$, as $A\cong C_n\ast C_2$. Suppose $n>2$. Now, $G=\<X; R^n\>$ and consider $G^{\prime}=\<X; R^{n/2}\>$. Then the image of $C$ in $G^{\prime}$ is trivial and so $G^{\prime}\cong A^{\prime}\ast B^{\prime}$ where $A^{\prime}\cong C_{n/2}$ and $B^{\prime}$ is non-trivial (as $\widehat{G^{\prime}}\cong\<X; R\>$ is non-trivial and $B^{\prime}$ surjects onto $\widehat{G^{\prime}}$). Thus, $G^{\prime}$ is infinitely ended. Then, because $\widehat{G^{\prime}}\cong \widehat{G}$ we can apply Lemma~\ref{lem:ends} to get that $G$ is infinitely ended, a contradiction. Thus, $n=2$ and $A$ is infinite dihedral, as required.

Combining the cases proves the result.
\end{proof}

We now give our classification of the ways in which a one-ended one-relator group with torsion can split as a free product with amalgamation over a virtually-$\mathbb{Z}$ subgroup. Recall that for $G=\<X; R^n\>$, $T:=\<\<R\>\>$ denotes the normal closure of the element $R$ and that a splitting $A\ast_CB$ is called {non-trivial} if $C\lneq A, B$.

\begin{lemma}
\label{lem:amalg}
Assume that $G=\<X; R^n\>$ is a one-ended one-relator group with torsion. If $G$ splits non-trivially as a free product with amalgamation $G\cong A\ast_CB$ over a virtually-$\mathbb{Z}$ subgroup $C$ then one of the following occurs.
\begin{enumerate}
\item $C\not\leq T$, and so $C$ is a subgroup of a malnormal, infinite cyclic subgroup of $G$.
\item Both $A$ and $C$ are infinite dihedral.
\item Both $B$ and $C$ are infinite dihedral.
\end{enumerate}
\end{lemma}

\begin{proof}
By Proposition~\ref{prop:vZsubgpsareZorD}, $C$ is either infinite cyclic or infinite dihedral. By Lemma~\ref{lem:vcsubgroups}, if $C\not\leq T$ then $C$ is a subgroup of a malnormal, infinite cyclic subgroup of $G$. On the other hand, suppose $C\leq T$. Then, by Lemma~\ref{lem:ZsplitFPA}, $C$ is infinite dihedral, and then, by Lemma~\ref{lem:DsplitFPA}, one of $A$ or $B$ is infinite dihedral.
\end{proof}

We now prove the Principal Lemma, which gives a description of the edge groups of the JSJ-decompositions of one-ended one-relator groups with torsion.

\begin{lemma}[Principal Lemma]
\label{thm:generalJSJstructural}
Suppose $G=\<X; R^n\>$ is a one-ended one-relator group with torsion. If $v$ is an elementary vertex in the JSJ-decomposition of $G$ then
one of the following occurs.
\begin{enumerate}
\item The vertex $v$ has arbitrary degree, and $G_v\not\leq T$. Thus, $G_v$ is a subgroup of a malnormal, infinite cyclic subgroup of $G$.
\item The vertex $v$ has degree one, and, writing $e$ for the incident edge, both $G_v$ and $G_e$ are infinite dihedral.
\end{enumerate}
\end{lemma}

\begin{proof}
Let $\Gamma$ be the graph underlying the JSJ-decomposition of $G$. Let $v$ be an arbitrary elementary vertex of $\Gamma$ and let $e$ be an edge incident to $v$. We shall prove that either $G_e\not\leq T$ (which implies $G_v\not\leq T$) or $v$ is an elementary vertex of degree one whose vertex group and adjacent edge group are each infinite dihedral. Combining the first case with Lemma~\ref{lem:vcsubgroups} proves the theorem. Recall that $G_e$ is virtually-$\mathbb{Z}$.

Suppose $e$ is \emph{not} a separating edge of $\Gamma$. Then $G$ splits as an HNN-extension over $G_e$, so $G=A\ast_{G_e^t=G_e^{\prime}}$. Thus, by Lemma~\ref{lem:HNN}, we have that $G_e\not\leq T$, as required.

Suppose $e$ is a separating edge of $\Gamma$. Then $G$ splits as a free product with amalgamation over $G_e$, so $G=A\ast_{G_e}B$ where $G_e\lneq A, B$. Thus, by Lemma~\ref{lem:amalg}, we have that either $G_e\not\leq T$ or $G_e$ and one of $A$ or $B$ is infinite dihedral. Suppose, without loss of generality, that $A$ is infinite dihedral. Then $A$ corresponds to an elementary vertex $v$ of degree one in the JSJ-decomposition, where $G_v=A$ is infinite dihedral and $v$ is a vertex of the edge $e$. This prove the lemma.
\end{proof}

\subsection{The proof of Theorem~\ref{thm:TheoremA}}
We now prove Theorem~\ref{thm:TheoremA}, as stated in the introduction to Section~\ref{sec:JSJ}. The proof follows quickly from the Principal Lemma, Lemma~\ref{thm:generalJSJstructural}.

\begin{proof}
[Theorem~\ref{thm:TheoremA}]
We shall prove that precisely one non-elementary vertex group of the JSJ-decomposition of $G$ contains torsion. That all other non-elementary vertex groups are torsion-free hyperbolic follows immediately as all vertex groups in a JSJ-decomposition are hyperbolic. The classification of elementary vertices follows from Lemma~\ref{thm:generalJSJstructural}.

Note that the root $R$ of the relator $R^n$ is contained in a conjugate of a vertex group $G_v$ in the JSJ-decomposition of $G$, and suppose that $g$ has finite order in $G$ but is contained in a different vertex group of the JSJ-decomposition, so $g\in G_w$ with $v\neq w$. Collapse each elementary vertex of degree one into its adjacent vertex. In the resulting graph of groups, $R$ is contained in a conjugate of a vertex group $G_{v^{\prime}}$, and the element $g$ of finite order is contained in a different vertex group, so $g\in G_{w^{\prime}}$ with ${v^{\prime}}\neq {w^{\prime}}$. Then, by Proposition~\ref{prop:conjToR}, $g$ is conjugate to some element of $G_{v^{\prime}}$, so $h^{-1}gh\in G_{v^{\prime}}$. Hence, $g\in G_{hv^{\prime}}$ with $h{v^{\prime}}\neq w$. However, then $g$ fixes the geodesic between $w$ and $h{v^{\prime}}$, and hence some edge stabiliser contains torsion, a contradiction.
\end{proof}

The possible existence of elementary vertices whose associated group is infinite dihedral in Theorem~\ref{thm:TheoremA}, which we shall call \emph{dihedral vertices}, is merely a technical quirk resulting from the definition of a JSJ-decomposition which we used (as opposed to the definition as a deformation space, which omits elementary vertices \cite{dahmani2011isomorphism}) and does not impinge on our partial resolution of the folk conjecture. Indeed, we do not know if there exist any one-relator group with torsion $G$ whose JSJ-decomposition contains dihedral vertices\footnote{See Question~\ref{Q:infDihedral} from the introduction.}, but given such a JSJ-decomposition one can sink each dihedral vertex into its unique adjacent vertex and the resulting decomposition is still canonical. Moreover, any dihedral vertices in the JSJ-decompositon of such a group $G$ are ignored when studying the outer automorphism group of $G$ \cite{levitt2005automorphisms} \cite{dahmani2011isomorphism} \cite{logan2014outer}, or when computing the isomorphism class of $G$ \cite{dahmani2011isomorphism}.

\section{Two-generator, one-relator groups with torsion}
\label{sec:2Gen1Rel}

In this section we state two theorems relating to the structure of non-Fuchsian one-ended two-generator, one-relator groups with torsion. Note that a two-generator, one-relator group with torsion $\langle a, b; R^n\rangle$ is Fuchsian if and only if $[a, b]$ is a cyclic shift of $R$ or $R^{-1}$ \cite{fine211classification} \cite{Pride1977}. The first of these theorems, Theorem~\ref{thm:JSJgraph}, gives, in a certain sense, the possibilities for the JSJ-decomposition of a non-Fuchsian one-ended two-generator, one-relator group with torsion. The second of these theorems, Theorem~\ref{thm:JSJdecomp}, applies Theorem~\ref{thm:JSJgraph} to prove that the outer automorphism groups of these groups are virtually-cyclic.

We have not included the proofs of Theorems~\ref{thm:JSJgraph}~and~\ref{thm:JSJdecomp} in this paper. This is because these theorems closely parallel results of Kapovich--Weidmann for two-generator torsion-free hyperbolic groups \cite{kapovich1999structure} and, because of Lemma~\ref{thm:generalJSJstructural}, the global structure of the proofs of Kapovich--Weidmann can be applied with only minor alterations. The altered proofs have been written out in full in the authors PhD thesis \cite{logan2014outer}.

Theorem~\ref{thm:JSJgraph} encodes the possibilities for the JSJ-decomposition of a two-generator, one-relator group with torsion. It describes the possible graph structure underlying a graph of groups decomposition of a two-generator, one-relator group with torsion $G$ where all edge stabilisers are virtually-$\mathbb{Z}$ and no vertex stabiliser is virtually-$\mathbb{Z}$. This encodes the JSJ-decomposition of the group $G$ because sinking the elementary vertices into adjacent vertices in the JSJ-decomposition yields a graph of groups where the underlying graph does not depend on how we dealt with the elementary vertices.

\begin{theorem}\label{thm:JSJgraph}
Assume that $G$ is a one-ended two-generator, one-relator group with torsion with $G\not\cong \< a, b; [a, b]^n\>$. Suppose that $G$ is the fundamental group of a graph of groups $\bGamma$ where all edge stabilisers are virtually-$\mathbb{Z}$ and no vertex-stabiliser is virtually-$\mathbb{Z}$. Then the graph $\Gamma$ underlying $\bGamma$ is one of the following.
\begin{enumerate}
\item A single vertex with no edges.
\item A single vertex with a single loop edge.
\end{enumerate}
Moreover, the vertex group is a two-generator, one-relator group with torsion.
\end{theorem}

Theorem~\ref{thm:JSJdecomp} applies Theorem~\ref{thm:JSJgraph} to prove that the outer automorphism groups of non-Fuchsian two-generator, one-relator groups with torsion are virtually-cyclic. The proof is, in essence, an application of results of Levitt \cite{levitt2005automorphisms} to Theorem~\ref{thm:JSJgraph}.

\begin{theorem}\label{thm:JSJdecomp}
Assume that $G$ is a one-ended two-generator, one-relator group with torsion. Then either $\out(G)$ is virtually-cyclic or $G\cong \< a, b; [a, b]^n\>$ for some $n>1$.
\end{theorem}

Using Theorem~\ref{thm:JSJdecomp} as a starting point, it can be shown that a one-ended two-generator, one-relator group with torsion has outer automorphism group isomorphic to one of only a finite number of groups, all of which are subgroup of $\operatorname{GL}_2(\mathbb{Z})$ \cite{Logan:2012} (see also the author's PhD thesis~\cite{logan2014outer}).

\section{Some examples of JSJ-decompositions}
\label{sec:examples}

We end the paper by giving some examples of one-relator groups with torsion which have non-trivial JSJ-decomposition. These examples demonstrate that the statement of Theorem~\ref{thm:TheoremA} is non-trivial. We first give a two-generator example, before demonstrating a general method for producing examples with arbitrary number of generators.

\begin{example}
\label{ex:2-genLoop}
Consider the following one-relator group with torsion.
\[
G=\<a, t; (t^{-1}a^{-1}ta^2)^2\>
\]
The group $G$ splits as an HNN-extension with stable letter $t$ and base group $\< b, c; (bc^2)^2\>\cong \mathbb{Z}\ast C_2$, the free product of the infinite cyclic group with the cyclic group of order two. By Theorem~\ref{thm:JSJgraph}, the vertices of degree greater than one of the JSJ-decomposition of $G$ consist of a single rigid vertex with vertex group $\mathbb{Z}\ast C_2$ and an elementary vertex with vertex group $\mathbb{Z}$. These vertices are connected with two edges, so that sinking the elementary vertex into the rigid vertex we obtain the aforementioned HNN-extension. This is illustrated in Figure~\ref{fig:2GenExample}.
\end{example}
		
\begin{figure}[ht]
\centering
\begin{tikzpicture}
\filldraw[black] (0, 1) circle (2pt);
\filldraw[black] (0, -1) circle (2pt);
\draw
(0, -1) .. controls (-1, -0.5) and (-1, 0.5) .. (0, 1);
\draw
(0, -1) .. controls (1, -0.5) and (1, 0.5) .. (0, 1);
\node at (-1.2, 0) {$\<e_1\>$};
\node at (1.2, 0) {$\<e_2\>$};
\node at (0, 1.3) {$\<x\>$};
\node at (0, -1.3) {$G_v=\<b, c; (bc^2)^2\>$};
\end{tikzpicture}
\caption{
\footnotesize Omitting elementary vertices of degree $1$, the JSJ-decomposition of the group $G=\<a, t; (t^{-1}a^{-1}ta^2)^2\>$ is the above graph of groups. Edge maps are given by $e_1\mapsto x$, by $e_2\mapsto x$, by $e_1\mapsto b$ and by $e_2\mapsto c^2$.
}\label{fig:2GenExample}
\end{figure}
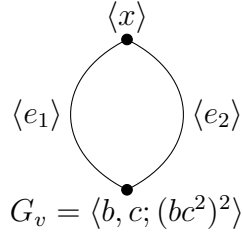

Example~\ref{ex:2-genLoop} is of interest because B.B.Newman demonstrated that the group $G$ contains a three-generated subgroup which is neither free nor a one-relator group with torsion~\cite{pride1977twogensubgps}. Therefore, this group was a possible counter-example to the folk conjecture which motivates this paper and, following our formalisation of the folk conjecture, a possible counter-example to Question~\ref{Q:freevertices} as stated in the introduction. The above demonstrates that the construction is, in fact, not a counter-example to our JSJ-decomposition interpretation of the folk-conjecture.

\p{More than two generators} The following example easily generalises to give a method for producing one-relator groups with torsion $G$ with an arbitrary number $m>2$ of generators which have non-trivial JSJ-decomposition.

\begin{example}
Consider the following one-relator group with torsion, where $n$ is odd.
\[
G=\<a, b, c, d; ([a^2, b^2]^2[c^2, d^2][a^2, b^2]^{-3}[c^2, d^2])^n\>
\]
The group $G$ splits as a free product with amalgamation as follows.
\[
G=\<a, b\>\ast_{[a^2, b^2]=x}\<x, y; (x^2yx^{-3}y)^n\>\ast_{y=[c^2, d^2]}\<c, d\>
\]
The middle group $\< x, y; (x^2yx^{-3}y)^n\>$ has finite outer automorphism group \cite[Lemma 6.2]{Logan:2012} so this splitting cannot be refined further. Therefore, adding in elementary vertices to split each edge yields the JSJ-decomposition of the group $G$. This is illustrated in Figure~\ref{fig:GeneralExample}.
\end{example}
This construction can be generalised by taking the relator to be the product of words over disjoint alphabets. Note that in a group obtained from this generalisation, one vertex group in the analogous splitting is a one-relator group with torsion, and all other vertex groups are free, by the Freiheitssatz \cite[Theorem 4.10]{mks}.
		
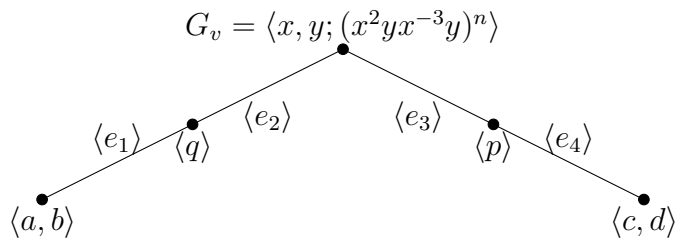
\begin{figure}[ht]
\centering
\begin{tikzpicture}
\filldraw[black] (0, 0) circle (2pt);
\filldraw[black] (2, -1) circle (2pt);
\filldraw[black] (4, -2) circle (2pt);
\filldraw[black] (-2, -1) circle (2pt);
\filldraw[black] (-4, -2) circle (2pt);
\draw
(0, 0) -- (2, -1);
\draw
(0, 0) -- (-2, -1);
\draw
(2, -1) -- (4, -2);
\draw
(-2, -1) -- (-4, -2);
\node at (0, 0.3) {$G_v=\<x, y; (x^2yx^{-3}y)^n\>$};
\node at (2, -1.3) {$\<p\>$};
\node at (4, -2.3) {$\<c, d\>$};
\node at (-2, -1.3) {$\<q\>$};
\node at (-4, -2.3) {$\<a, b\>$};
\node at (-3, -1.2) {$\<e_1\>$};
\node at (-1, -0.9) {$\<e_2\>$};
\node at (1, -0.9) {$\<e_3\>$};
\node at (3, -1.2) {$\<e_4\>$};
\end{tikzpicture}
\caption{
\footnotesize
This graph of groups is the JSJ-decomposition of $G=\<a, b, c, d; ([a^2, b^2]^2[c^2, d^2][a^2, b^2]^{-3}[c^2, d^2])^n\>$, $n>1$ odd. Edge maps are given by $e_1\mapsto [a^2, b^2]$, by $e_2\mapsto x$, by $e_3\mapsto y$, and by $e_4\mapsto [c^2, d^2]$; all other edge maps are trivial.
}\label{fig:GeneralExample}
\end{figure}

{\footnotesize\p{Acknowledgements}
The author would like to thank his PhD supervisor, Stephen J. Pride, and Tara Brendle for many helpful discussions about this paper, and Jim Howie for suggestions regarding a preprint. He would also like to thank an anonymous referee of another paper \cite{Logan:2012} for the suggestion to apply the ideas of Kapovich--Weidmann to prove Theorem~\ref{thm:JSJdecomp}, which led to this paper.}

\bibliographystyle{amsalpha}
\bibliography{BibTexBibliography}

\end{document}